\documentclass[a4paper,twoside,12pt]{article}

\usepackage{amssymb,amsmath,amsthm,latexsym}
\usepackage{amsfonts, centernot}
\usepackage{graphicx,caption,subcaption}
\usepackage[pdftex,bookmarks,colorlinks=false]{hyperref}
\usepackage[hmargin=1.2in,vmargin=1.2in]{geometry} 
\usepackage[mathscr]{euscript}

\usepackage[auth-lg,affil-sl]{authblk}
\allowdisplaybreaks

\usepackage{float}
\usepackage{calc,tikz}
\usetikzlibrary{arrows,decorations.markings}
\tikzstyle{vertex}=[circle, draw, inner sep=1pt, minimum size=8pt]

\usepackage{multirow}
\usepackage{hhline}
\usepackage{booktabs}
\usepackage{longtable}

\def\noi{\noindent}

\def\cC{\mathcal{C}}

\newcommand{\Mod}[1]{\ (\mathrm{mod}\ #1)}

\allowdisplaybreaks
\newtheorem{theorem}{Theorem}[section]

\newtheorem{definition}{Definition}[theorem]

\title{\textbf{\sc Chromatic Topological Indices of Certain Cycle Related Graphs}}

\author{Smitha Rose}
\affil{\small Department of Mathematics\\ Christ University \\ Bengaluru, Karnataka, India.\\ {\tt smitha.rose@res.christuniversity.in}}

\author{Sudev Naduvath}
\affil{\small Centre for Studies in Discrete Mathematics,\\ Vidya Academy of Science \& Technology, \\ Thrissur-680501, Kerala, India.\\ {\tt sudevnk@gmail.com}}

\date{}

\begin{document}
\maketitle

\begin{abstract}

Topological indices are real numbers invariant under graph isomorphisms. Chromatic analogue of topological indices have been introduced recently in literature in 2017. 
Mainly, chromatic versions of  Zagreb indices are studied lately. This paper discusses the notion of chromatic topological and irregularity indices of certain cycle related graphs.

\end{abstract}

\vspace{0.2cm}

\noi \textbf{Key words}: Chromatic Zagreb index,  chromatic irregularity index

\vspace{0.25cm}

\noi \textbf{Mathematics Subject Classification 2010}: 05C15,05C38.


\section{Introduction} 

Chemical graph theory finds a variety of application in today's world especially in the field of pharmaceuticals. It concerns itself mainly with the mathematical modelling of the chemical phenomenon and the gaining of valuable insights into chemical behaviour. A \textit{topological index} of a graph $G$ is a real number preserved under isomorphisms, which constitutes one of the basic notions of molecular descriptors in chemical graph theory. \textit{Chromatic topological indices} of a graph $G$ is a term that has recently been coined \cite{KSM} to designate a new colouring version to these indices which embrace both proper colouring and topological indices. Here, the vertex degrees are interchanged with minimal colouring, keeping up the additional colouring conditions of proper colouring. The graphs discussed in this paper are finite, non-trivial, undirected, connected and without loops or multiple edges. For notation and terminology not explicitly defined here see \cite{FH,BM1,Cha,DBW}.

Generally, graph colouring is referred to as an assignment of colours, labels or weights to the vertices of a graph under consideration subject to certain conditions. A \textit{proper vertex colouring} of a graph $G$ is an assignment $\varphi:V(G)\to \cC$  of the vertices of $G$, where $\cC= \{c_1,c_2,c_3,\ldots,c_\ell\}$ is a set of colours such that adjacent vertices of $G$ are given different colours. The minimum number of colours required to apply a proper vertex colouring to $G$ is called the \textit{chromatic number} of $G$ and is denoted $\chi(G)$. The set of all vertices of $G$ which have the colour $c_i$ is called the \textit{colour class} of that colour $c_i$ in $G$ and may be denoted by $\cC_i$. The \textit{strength} of the colour class, denoted by $\theta(c_i)$, is the cardinality of each colour class of colour $c_i$. 

A proper vertex colouring consisting of the colours having minimum subscripts may be called a \textit{minimum parameter colouring} (see \cite{KSM}). If we colour the vertices of $G$ in such a way that $c_1$ is assigned to maximum possible number of vertices, then
$c_2$ is assigned to maximum possible number of remaining uncoloured vertices and proceed in this manner until all vertices are coloured, then such a colouring is called \textit{$\varphi^-$-colouring} of $G$. In a similar manner, if $c_\ell$ is assigned to maximum possible number of vertices, then $c_{(\ell-1)}$ is assigned to maximum possible number of remaining uncoloured vertices and proceed in this manner until all vertices are coloured, then such a colouring is called \textit{$\varphi^+$-colouring} of $G$ (see \cite{KSM}). 

For computational convenience, we define function $\zeta:V(G)\to \{1,2,3,\ldots,\ell\}$ such that $\zeta(v_i)=s\iff \varphi(v_i)=c_s,c_s\in \cC$. The total number of edges with end points having colours $c_t$ and $c_s$ is denoted by $\eta_{ts}$, where $t<s,1 \leq t,s \leq \chi(G)$.

Analogous to the definitions of Zagreb and irregularity indices of graphs (see \cite{Abd,Gut,Zho,Zhou}), the notions of different chromatic Zagreb indices and chromatic irregularity indices have been introduced in \cite{KSM} as follows:

\begin{definition}{\rm 
\cite{KSM} Let $G$ be a graph and let  $\cC=\{c_1,c_2,c_3,\ldots,c_\ell\}$ be a proper colouring of $G$ such that $\varphi(v_i)=c_s; 1\le i\le n, 1\le s\le \ell$. Then for $1\le t \le \ell\,!$, 
\begin{enumerate}\itemsep0mm 
\item[(i)]~ The \textit{first chromatic Zagreb index} of $G$, denoted by $M^{\varphi_t}_1(G)$, is defined as $M^{\varphi_t}_1(G)=\sum\limits_{i=1}^{3}(\zeta(v_i))^2=\sum\limits_{j=1}^{\ell}\theta(c_j)\cdot j^2$.
\item[(ii)]~ The \textit{second chromatic Zagreb index} of $G$, denoted by $M^{\varphi_t}_2(G)$, is defined as $M^{\varphi_t}_2(G)=\sum\limits_{i=1}^{n-1}\sum\limits_{j=2}^{n}(\zeta(v_i)\cdot \zeta(v_j)),\ v_iv_j \in E(G)$.
\item[(iii)]~ The \textit{chromatic irregularity index} of $G$, denoted by $M^{\varphi_t}_3(G)$, is defined as 
$ M^{\varphi_t}_3(G)= \sum\limits_{i=1}^{n-1}\sum\limits_{j=2}^{n}|\zeta(v_i)-\zeta(v_j)|,\ v_iv_j \in E(G)$. 
\end{enumerate}  
}\end{definition}

\begin{definition} \cite{RoSu}
The \textit{chromatic total irregularity index} of a graph $G$ corresponding to a proper colouring $\varphi:V(G)\to \cC=\{c_1,c_2, \ldots c_{\ell}\}$ is defined as 
$$M^{\varphi_t}_4(G)=\frac{1}{2}\sum\limits_{i=1}^{n-1}\sum\limits_{j=2}^{n}|\zeta(v_i)-\zeta(v_j)|,\ v_i,v_j \in V(G).$$
\end{definition}

In view of the above notions, the minimum and maximum chromatic Zagreb indices and the corresponding irregularity indices are defined in \cite{KSM} as follows.

\begin{eqnarray*}
M^{\varphi^-}_i(G) & = & \min\{M^{\varphi_t}_1(G): 1\le t \le \ell!\}, for 1 \leq i \leq 4 \\
M^{\varphi^+}_i(G) & = & \max\{M^{\varphi_t}_1(G): 1\le t \le \ell!\},for 1 \leq i \leq 4 . \\
\end{eqnarray*}


\section{Chromatic Topological Indices of Flower Graphs}

A \textit{flower graph} $F_n$ is a graph which is obtained by joining the pendant vertices of a helm graph $H_n$ to its central vertex. The following theorem discusses the chromatic topological indices of the flower graph $F_n$.

\begin{theorem}\label{Thm-2.1}
For a flower graph $F_n$, we have
\begin{enumerate}
\item[(i)]~ $M_{1}^{\varphi^{-}}(F_n)=
\begin{cases}
5n+9; & \text{if $n$ is even};\\
5n+21; & \text{if $n$ is odd}.
\end{cases}$
\item[(ii)]~ $M_{2}^{\varphi^{-}}(F_n)= 
\begin{cases}
13n; & \text{if $n$ is even};\\
13n+16; & \text{if $n$ is odd}.
\end{cases}$
\item [(iii)]~ $M_{3}^{\varphi^{-}}(F_n)= 
\begin{cases}
5n; & \text{if $n$ is even};\\
5(n+1); & \text{if $n$ is odd}.
\end{cases}$
\item[(iv)]~ $M_{4}^{\varphi^{-}}(F_n)=
\begin{cases}
\frac{n^2+3n}{2}; & \text{if $n$ is even};\\
\frac{n^2+7n-2}{2}; & \text{if $n$ is odd}.
\end{cases}$
\end{enumerate}
\end{theorem}
\begin{proof}~
The chromatic number of a flower graph $F_n$ is $3$ when $n$ is even and is $4$ when $n$ is odd. Let $v_1, v_2, \ldots v_n$ be the non-adjacent vertices around $C_n$ and $u_1, u_2, \ldots u_n$ be the vertices of $C_n$  on the rim of the wheel and $u$ be the central vertex. In order to calculate the minimum values of chromatic Zagreb indices, we apply the $\varphi^{-}$ colouring pattern to $F_n$ as described  below.

If $n$ is even, then the sets $S_1=\{v_1, v_3, \ldots v_{n-1}, u_2, u_4, \ldots u_n\}$ and $S_2=\{v_2,v_4, \ldots, v_{n}, u_1, u_3, \ldots u_{n-1}\}$ form the two maximum independent sets with same cardinality $n$. We colour them with minimum colours $c_1$ and $c_2$ respectively. The remaining central vertex $u$ is coloured with $c_3$. If $n$ is odd, then the maximum independent sets $S_1=\{v_1, v_3, \ldots v_n, u_2, u_4, \ldots u_{n-1}\}$, $S_2=\{v_2, v_4, \ldots v_{n-1}, u_1, u_3, \ldots u_{n-2}\}$ have cardinality $n$ and $n-1$ and coloured with $c_1$ and $c_2$ respectively. Also, the vertices $u_n$ and $u$ are coloured with  $c_4$ and $c_3$ respectively. Then,

\noi \textit{Part (i):} In order to calculate $M_{1}^{\varphi^{-}}$ of $F_n$, we first colour the vertices as mentioned above and then proceed to consider the following cases.

\textit{Case-1:} Let $n$ be even, then we have $\theta(c_1)=\theta(c_2)=n$ and $\theta(c_3)= 1$. Therefore, the corresponding chromatic Zagreb index is given by
$$M^{\varphi^-}_1(F_n)=\sum\limits_{i=1}^{3}(\zeta(v_i))^2=5n+9.$$

\noi \textit{Case-2:} Let $n$ be odd. Then, we have $\theta(c_1)=n, \theta(c_2)=n-1$ and $\theta(c_3)=\theta(c_4)=1.$ Now, by the definition of first chromatic Zagreb index, we have $$M^{\varphi^-}_1(F_n)=\sum\limits_{i=1}^{4}(\zeta(v_i))^2= 5n+21.$$

\noi \textit{Part (ii):} To compute $M_{2}^{\varphi^{-}}$, we first colour the vertices as per the instructions in introductory part for both cases of $n$. Now, consider the following cases:

\noi \textit{Case-1:} If $n$ is even, then we observe that $\eta_{12}=2n, \eta_{23}=\eta_{13}=n$. The definition of second chromatic Zagreb index gives the sum $$M^{\varphi^-}_2(F_n)=\sum\limits_{1 \leq t,s \leq \chi(F_n)}^{t<s}ts\eta_{ts}=4n+3n+6n=13n.$$  

\noi \textit{Case- 2:} Let $n$ be odd. Here we see that $\eta_{12}=2n-3, \eta_{13}=n, \eta_{23}=n-1, \eta_{14}= 2,\\ \eta_{34}=\eta_{24}=1.$ Hence, we have the sum $$M^{\varphi^-}_2(F_n)=\sum\limits_{1 \leq t,s \leq \chi(F_n)}^{t<s}ts\eta_{ts}= 13n+16.$$

\noi \textit{Part (iii):} We calculate the minimum irregularity measurement by considering the following cases:

\noi \textit{Case- 1:} Let $n$ be even. Here, $\eta_{12}+ \eta_{23}=3n$ edges contribute the distance $1$ to the total summation, while $\eta_{13}=n$ contribute the distance $2$. The result follows from the following calculations: 
$$ M^{\varphi^-}_3(F_n)= \sum\limits_{i=1}^{n-1}\sum\limits_{j=2}^{n}|\zeta(v_i)-\zeta(v_j)|=5n.$$ 

\noi \textit{Case- 2:} Let $n$ be odd. Here we see that $\eta_{12}+\eta_{23}+ \eta_{34}$ edges contribute $1$ to the colour distance, $\eta_{13}+\eta_{24}$ edges contribute $2$, while $\eta_{14}$ edges contribute $3$.  Then the result follows from the following calculations: $$M^{\varphi^-}_3(F_n)= \sum\limits_{i=1}^{n-1}\sum\limits_{j=2}^{n}|\zeta(v_i)-\zeta(v_j)|=5(n+1).$$ 

\noi \textit{Part (iv):} To calculate the chromatic total irregularity indices of flower graphs, we have to consider all the possible vertex pairs and all colour combinations contributing non zero distances are considered according to the following two cases:\\
\noi \textit{Case- 1:} Let $n$ be even. The combinations possible are charted as $\{1,2\}, \{2,3\}$ contributing a distance $1$ and $\{1,3\}$ contributing $2$. Observe that $\theta(c_1)=\theta(c_2)=n$ and $\theta(c_3)=1$. Thus, we have

\begin{eqnarray*}
M^{\varphi^-}_4(F_n)& = & \frac{1}{2}\sum\limits_{u,v \in V(F_n)}|\zeta(u)-\zeta(v)|\\ 
& = & =\frac{n^2+3n}{2}.
\end{eqnarray*}
\noi \textit{Case- 2:} Let $n$ be odd. Here, the possible combinations which contribute to the colour distances are $\{1,2\}, \{2,3\},\{3,4\}$ contributing $1$, $\{1,3\}, \{2,4\}$ contributing $2$ and $\{1,4\}$ contributing $3$. We calculate the chromatic total irregularity as given below:

\begin{eqnarray*}
M^{\varphi^-}_4(F_n) & = & \frac{1}{2}\sum\limits_{u,v \in V(F_n)}|\zeta(u)-\zeta(v)|\\
& = &=\frac{n^2+7n-2}{2}
\end{eqnarray*}
\end{proof}


Instead of $\varphi^-$ colouring, one can also work with $\varphi^+$ colouring of flower graphs using minimum parameter colouring. The results obtained are charted below as next theorem.

\begin{theorem}\label{Thm-2.2}
For a flower graph $F_n$, we have 
\begin{enumerate}
\item[(i)]~ $M_{1}^{\varphi^{+}}(F_n)=
\begin{cases}
13n+1; & \text{if $n$ is even}\\
25n-4; & \text{if $n$ is odd};
\end{cases}$
\item[(ii)]~ $M_{2}^{\varphi^{+}}(F_n)= 
\begin{cases}
17n; & \text{if $n$ is even}\\
38n-29; & \text{if $n$ is odd};
\end{cases}$
\item [(iii)]~ $M_{3}^{\varphi^{+}}(F_n)= 
\begin{cases}
5n; & \text{if $n$ is even}\\
5n+5; & \text{if $n$ is odd};
\end{cases}$
\item[(iv)]~ $M_{4}^{\varphi^{+}}(F_n)=
\begin{cases}
\frac{n^2+3n}{2}; & \text{if $n$ is even}\\
\frac{n^2+7n-2}{2}; & \text{if $n$ is odd}.
\end{cases}$
\end{enumerate}
\end{theorem}
\begin{proof}~
Here we follow $\varphi_+$ colouring of flower graphs to obtain desired results.

\noi When $n$ is even,  the vertices $S_1=\{v_1, v_3, \ldots v_{n-1}, u_2, u_4, \ldots u_n\}$ and $S_2=\{v_2, v_4, \ldots \\v_{n}, u_1, u_3, \ldots u_{n-1}\}$ forms the two maximum independent sets with same cardinality $n$. We colour them with maximum colours $c_3$ and $c_2$ respectively. The remaining central vertex $u$ is coloured with $c_1$. Then $\eta_{12}=\eta_{13}=n$ and $\eta_{23}=2n$.

\noi If $n$ is odd, we have chromatic number $4$. Here the maximum independent sets $S_1=\{v_1, v_3, \ldots v_n, u_2, u_4, \ldots u_{n-1}\}$, $S_2=\{v_2, v_4, \ldots v_{n-1}, u_1, u_3, \ldots u_{n-2}\}$ have cardinality $n$ and $n-1$ and coloured with $c_1$ and $c_2$ respectively. Also, the vertices $u_n$ and $u$ are coloured with  $c_4$ and $c_3$ respectively. to get maximum values. Thus $\eta_{12}=\eta_{13}=1, \eta_{14}=2$, $\eta_{23}=n-1, \eta_{24}=n$ and $\eta_{34}=2n-3$.

\noi The remaining part of the proof follows in the same way as mentioned in the proof of Theorem \ref{Thm-2.1}.
\end{proof}


\section{Chromatic Topological Indices of Sunflower Graphs}

A \textit{sunflower graph} $SF_n$ is a graph obtained by replacing each edge of the rim of a wheel graph $W_n$ by a triangle such that two triangles share a common vertex if and only if the corresponding edges in $W_n$ are adjacent in $W_n$. The following result discusses the chromatic topological indices of a sunflower graph by using $\varphi_-$ colouring.


\begin{theorem}\label{Thm-3.1}
For a sun flower graph $SF_n$, we have 
\begin{enumerate}
\item[(i)]~ $M_{1}^{\varphi^{-}}(SF_n)=
\begin{cases}
\frac{15n+2}{2}; & \text{if $n$ is even}\\
\frac{15n+21}{2}; & \text{if $n$ is odd};
\end{cases}$
\item[(ii)]~ $M_{2}^{\varphi^{-}}(SF_n)= 
\begin{cases}
\frac{27n}{2}; & \text{if $n$ is even}\\
\frac{27n+25}{2}; & \text{if $n$ is odd};
\end{cases}$
\item [(iii)]~ $M_{3}^{\varphi^{-}}(SF_n)= 
\begin{cases}
\frac{11n}{2}; & \text{if $n$ is even}\\
\frac{11n+11}{2}; & \text{if $n$ is odd};
\end{cases}$
\item[(iv)]~ $M_{4}^{\varphi^{-}}(SF_n)=
\begin{cases}
\frac{7n^2+6n}{8}; & \text{if $n$ is even}\\
\frac{7n^2+16n+1}{8}; & \text{if $n$ is odd}.
\end{cases}$
\end{enumerate}
\end{theorem}


\begin{proof}~

As we know, the sunflower graph $SF_n$ has chromatic number $3$ when $n$ is even and chromatic number $4$ when $n$ is odd. Let $v_1, v_2, \ldots v_n$ be the non-adjacent vertices on the outer cycle, $u_1, u_2, \ldots u_n$ be the vertices on the wheel and $u$ be the central vertex.  To obtain the minimum values of the chromatic topological indices we follow the $ \varphi^{-}$ colouring pattern to $SF_n$ as described  below.\\
When $n$ is even, we can find the maximum independent sets $S_1=\{v_1, v_2, \ldots v_n, u\}$ with cardinality $n+1$,  $S_2=\{ u_1, u_3, \ldots u_{n-1}\}$, $S_3=\{ u_2, u_4, \ldots u_{n}\}$ with same cardinality $\frac{n}{2}$. We colour them with minimum colours $c_1$, $c_2$ and $c_3$ respectively. 
If $n$ is odd, then $S_1=\{v_1, v_2, \ldots v_n, u\}$ with cardinality $n+1$ and we colour it with $c_1$. Now, $S_2=\{ u_1, u_3, \ldots u_{n-2}\}$, $S_3=\{ u_2, u_4, \ldots u_{n-1}\}$ are next maximum independent sets with same cardinality $\frac{n-1}{2}$. So we colour it with $c_2$ and $c_3$ respectively and the colour $c_4$ is assigned to vertex $u_n$. 

\noi \textit{Part (i):}  In order to find $M_{1}^{\varphi^{-}}$ of $SF_n$, we first colour the vertices as mentioned above and then proceed to consider the following cases:

\textit{Case-1:} Let $n$ be even.Then, we have $\theta(c_1)= n+1$ and $\theta(c_2)=\theta(c_3)=\frac{n}{2}$. Therefore, the corresponding chromatic Zagreb index  is given by
$$M^{\varphi^-}_1(SF_n)=\sum\limits_{i=1}^{3}(\zeta(v_i))^2=(n+1)+\frac{4n}{2}+\frac{9n}{2}=\frac{15n+2}{2}.$$

\noi \textit{Case-2:} Let $n$ be odd. Then, we have $\theta(c_1)=n+1, \theta(c_2)=\theta(c_3)=\frac{n-1}{2}$ and $\theta(c_4)=1.$ Now, by the definition of first chromatic Zagreb index, we have $$M^{\varphi^-}_1(SF_n)=\sum\limits_{i=1}^{4}(\zeta(v_i))^2=(n+1)+2(n-1)+\frac{9(n-1)}{2}+16= \frac{15n+21}{2}.$$

\noi \textit{Part (ii):} We colour the vertices as per the instructions in introductory part for even and odd cases of $n$. Now consider the following cases:

\noi \textit{Case-3:} Let $n$ be even. In this case, we see that $\eta_{23}=n, \eta_{12}=\eta_{13}=\frac{3n}{2}$. The definition of second chromatic Zagreb index gives the sum $$M^{\varphi^-}_2(SF_n)=\sum\limits_{1 \leq t,s \leq \chi(SF_n)}^{t<s}ts\eta_{ts}=\frac{6n}{2}+\frac{9n}{2}+6n=\frac{27n}{2}.$$  

\noi \textit{Case-4:} Let $n$ be odd. Here we see that $\eta_{12}= \eta_{13}= \frac{3(n-1)}{2}, \eta_{23}=n-2, \eta_{14}= 3,\\ \eta_{34}=\eta_{24}=1.$ Hence, we have the sum $$M^{\varphi^-}_2(SF_n)=\sum\limits_{1 \leq t,s \leq \chi(SF_n)}^{t<s}ts\eta_{ts}=3(n-1)+\frac{9(n-1)}{2}+6(n-2)+32=\frac{27n+25}{2}.$$

\noi \textit{Part (iii):} To find the minimum irregularity measurement, consider the following cases:
\noi \textit{Case-5:} Let $n$ be even. Here $\eta_{12}+ \eta_{23}=\frac{5n}{2}$ edges contribute the distance $1$ to the total summation while $\eta_{13}=\frac{3n}{2}$ contribute the distance $2$. The result follows from the following calculations:
$$ M^{\varphi^-}_3(SF_n)= \sum\limits_{i=1}^{n-1}\sum\limits_{j=2}^{n}|\zeta(v_i)-\zeta(v_j)|=\frac{3n}{2}+\frac{6n}{2}+n=\frac{11n}{2}.$$ 
\noi \textit{Case-6:} Let $n$ be odd. Here we see that, $\eta_{12}+\eta_{23}+ \eta_{34}$ edges contribute $1$ to the colour distance, $\eta_{13}+\eta_{24}$ edges contribute $2$, while $\eta_{14}$ edges contribute $3$.  Then, the result follows from the following calculations: $$M^{\varphi^-}_3(SF_n)= \sum\limits_{i=1}^{n-1}\sum\limits_{j=2}^{n}|\zeta(v_i)-\zeta(v_j)|= \frac{3(n-1)}{2}+3(n-1)+(n-2)+12=\frac{11n+11}{2}.$$ 

\noi \textit{Part (iv):}  To calculate the total irregularity of $SF_n$, all the possible vertex pairs from $SF_n$ have to be considered and their possible colour distances are determined. The possibility of the vertex pairs which contribute to the colour distance can be classified according to the following two cases.

\noi \textit{Case- 7:} Let $n$ be even. The combinations possible are charted as $\{1,2\}, \{2,3\}$ contributing $1$ and $\{1,3\}$ contributing $2$. Observe that $\theta(c_2)=\theta(c_3)=\frac{n}{2}$ and $\theta(c_1)=n+1$. Thus, we have
\begin{eqnarray*}
M^{\varphi^-}_4(SF_n)& = & \frac{1}{2}\sum\limits_{u,v \in V(SF_n)}|\zeta(u)-\zeta(v)|\\ 
& = & \frac{7n^2+6n}{8}.
\end{eqnarray*}
\noi \textit{Case-8:} Let $n$ be odd. Here, the possible combinations which contributes to the colour distances are $\{1,2\}, \{2,3\}, \{3,4\}$ contributing $1$, $\{1,3\}, \{2,4\}$ contributing $2$ and $\{1,4\}$ contributing $3$. We calculate the total irregularity as given below:

\begin{eqnarray*}
M^{\varphi^-}_4(SF_n) & = & \frac{1}{2}\sum\limits_{u,v \in V(SF_n)}|\zeta(u)-\zeta(v)|\\
& = & \frac{7n^2+16n+1}{8}
\end{eqnarray*}

\end{proof}

Using the minimum parameter colouring we can also work on $\varphi_+$ colouring  of sunflower graphs. Next theorem deals with this matter.

\begin{theorem}\label{Thm-3.2}
For a sun flower graph $SF_n$, we have 
\begin{enumerate}
\item[(i)]~ $M_{1}^{\varphi^{+}}(SF_n)=
\begin{cases}
\frac{23n+18}{2}; & \text{if $n$ is even}\\
\frac{45n+21}{2}; & \text{if $n$ is odd};
\end{cases}$
\item[(ii)]~ $M_{2}^{\varphi^{+}}(SF_n)= 
\begin{cases}
17n; & \text{if $n$ is even}\\
36n-25; & \text{if $n$ is odd};
\end{cases}$
\item [(iii)]~ $M_{3}^{\varphi^{+}}(SF_n)= 
\begin{cases}
\frac{11n}{2}; & \text{if $n$ is even}\\
\frac{11n+11}{2}; & \text{if $n$ is odd};
\end{cases}$
\item[(iv)]~ $M_{4}^{\varphi^{+}}(SF_n)=
\begin{cases}
\frac{7n^2+6n}{8}; & \text{if $n$ is even}\\
\frac{7n^2+16n+1}{8}; & \text{if $n$ is odd}.
\end{cases}$
\end{enumerate}
\end{theorem}
\begin{proof}~
Here we follow $\varphi_+$ colouring of sun flower graphs to obtain desired results.

\noi When $n$ is even, we have $\theta(c_3)= n+1$ and $\theta(c_2)=\theta(c_1)=\frac{n}{2}$. Then $\eta_{23}=\eta_{13}=\frac{3n}{2}$ and $\eta_{12}=n$.

\noi Let $n$ be odd, we have chromatic number $4$. Here the cardinality of the colour classes are$\theta(c_4)= n+1$,  $\theta(c_2)=\theta(c_3)=\frac{n-1}{2}$ and $\theta(c_1)= 1$.  Thus we have $\eta_{12}=\eta_{13}=1, \eta_{14}=3$, $\eta_{34}=\eta_{24}=\frac{3(n-1)}{2}$ and $\eta_{23}=n-2$.

\noi The remaining part of the proof follows in a similar manner as mentioned in the proof of Theorem \ref{Thm-3.1} 
\end{proof}


\section{Chromatic Topological Indices of Closed Sunflower Graphs}

A \textit{closed sunflower graph} $CSF_n$ is a graph obtained by joining the independent vertices of a sunflower graph $SF_n$, which are not adjacent to its central vertex so that these vertices induces a cycle on $n$ vertices. 
The following result provides the expressions for the topological indices of a closed sunflower graph.

\begin{theorem}\label{Thm-4.1}
For a closed sunflower graph $CSF_n$, we have
\begin{enumerate}
\item[(i)] $M_{1}^{\varphi^{-}}(CSF_n)= 
\begin{cases}
\frac{28n+48}{3}; &n \equiv 0 \Mod{3}\\
\frac{28n+143}{3}; &n \equiv 1 \Mod{3}\\
\frac{28n+109}{3}; &n \equiv -1 \Mod{3};
\end{cases}$
\item[(ii)] $M_{2}^{\varphi^{-}}(CSF_n)= 
\begin{cases}
\frac{68n}{3}; &n \equiv 0 \Mod{3}\\
\frac{74n+136}{3}; &n \equiv 1 \Mod{3}\\
\frac{68n+134}{3}; &n \equiv -1 \Mod{3};
\end{cases}$
\item[(iii)] $M_{3}^{\varphi^{-}}(CSF_n)= 
\begin{cases}
\frac{22n}{3}; &n \equiv 0 \Mod{3}\\
\frac{25n+8}{3}; &n \equiv 1 \Mod{3}\\
\frac{22n+10}{3}; &n \equiv -1 \Mod{3};
\end{cases}$
\item[(iv)] $M_{4}^{\varphi^{-}}(CSF_n)= 
\begin{cases}
\frac{16n^2+36n}{18}; &n \equiv 0 \Mod{3}\\
\frac{16n^2+94n-92}{18}; &n \equiv 1 \Mod{3}\\
\frac{16n^2+74n-32}{18}; &n \equiv -1 \Mod{3};
\end{cases}$
\end{enumerate}
\end{theorem}
\begin{proof}~
A closed sunflower graph $CSF_n$ has chromatic number $4$ when $n \equiv 0 \Mod{3}$ and has chromatic number $5$ otherwise. Let $v_1, v_2, \ldots v_n$ be the vertices of the inner wheel and $u_1, u_2, \ldots u_n$ be the vertices on the rim of the outer wheel and $v$ be the central vertex. In order to calculate the minimum values of chromatic Zagreb indices we apply the $\varphi^{-}$ colouring pattern to $CSF_n$ as described below:

Let $n \equiv 0 \Mod{3}$ be assumed. When $n \equiv 0 \Mod{3}$, obeying the rules of minimum colouring, we can find three colour classes with same cardinality $\frac{2n}{3}$ and we colour them with minimal colours $c_1, c_2, c_3$. The central vertex $v$ is coloured with $c_4$.

Let $n \equiv 1 \Mod{3}$ be assumed. Here we form three colour classes with the maximum independent sets having the same cardinality $\frac{2(n-1)}{3}$ and we colour them with minimal colours $c_1, c_2, c_3$. Also the vertices $u_1$ and $v_n$ are coloured with  $c_4$ and the central vertex $v$ is coloured with $c_5$. 

Let $n \equiv -1 \Mod{3}$ be assumed. Here we form three colour classes with the maximum independent sets having the same cardinality $\frac{2n-1}{3}$ and we colour them with minimal colours $c_1, c_2, c_3$. Also the balance two vertices $v$ and $v_n$ are coloured with  $c_4$ and $c_5$ respectively. Now we can deal with the parts of the proof. 

\noi \textit{Part (i):}  In order to calculate $M_{1}^{\varphi^{-}}$ of $CSF_n$, we first colour the vertices as mentioned above and then proceed to consider the following cases.

\textit{Case-1.1:} Let $n \equiv 0 \Mod{3}$. Then, we have $\theta(c_1)=\theta(c_2)=\theta(c_3)=\frac{2n}{3}$ and $\theta(c_4)= 1$. Therefore, the corresponding chromatic Zagreb index  is given by
$$M^{\varphi^-}_1(CSF_n)=\sum\limits_{i=1}^{4}(\zeta(v_i))^2=\frac{2n}{3}+\frac{8n}{3}+\frac{18n}{3}+16=\frac{28n+48}{3}.$$

\noi \textit{Case-1.2:} Let $n \equiv 1 \Mod{3}$. Then, we have $\theta(c_1)=\theta(c_2)=\theta(c_3)=\frac{2(n-1)}{3}$,  $\theta(c_4)= 2$ and $\theta(c_5)= 1$. Therefore, the corresponding chromatic Zagreb index  is given by
$$M^{\varphi^-}_1(CSF_n)=\sum\limits_{i=1}^{5}(\zeta(v_i))^2=25+32+\frac{18(n-1)}{3}+\frac{8(n-1)}{3}+\frac{2(n-1)}{3}=\frac{28n+143}{3}.$$

\noi \textit{Case-1.3:} Let $n \equiv -1 \Mod{3}$. Then we have $\theta(c_1)=\theta(c_2)=\theta(c_3)=\frac{2n-1}{3}$ and $\theta(c_5)= \theta(c_4)= 1$. Therefore, the corresponding chromatic Zagreb index  is given by
$$M^{\varphi^-}_1(CSF_n)=\sum\limits_{i=1}^{5}(\zeta(v_i))^2=\frac{28(n-1)}{3}+41=\frac{28n+109}{3}.$$

\noi \textit{Part (ii):} We colour the vertices as per the instructions in introductory part for different cases of $n$. Now consider the following cases:

\textit{Case-2.1:} Let $n \equiv 0 \Mod{3}$.  Here, we see that $\eta_{12}= \eta_{23}=\eta_{13}=\frac{4n}{3}$ and $\eta_{14}= \eta_{24}=\eta_{34}=\frac{n}{3}$. The definition of second chromatic Zagreb index gives the sum $$M^{\varphi^-}_2(CSF_n)=\sum\limits_{1 \leq t,s \leq \chi(CSF_n)}^{t<s}ts\eta_{ts}=\frac{44n}{3}+\frac{24n}{3}=\frac{68n}{3}.$$  

\noi \textit{Case-2.2:} Let $n \equiv 1 \Mod{3}$. Here, we see that $\eta_{12}= \eta_{23}=\frac{4n-7}{3}$, $\eta_{13}=\frac{4n-10}{3}$, $\eta_{15}= \eta_{25}=\eta_{35}=\frac{n-1}{3}$, $\eta_{14}= \eta_{34}=3$, $ \eta_{24}=2$  and $\eta_{45}=1$. The definition of second chromatic Zagreb index gives the sum $$M^{\varphi^-}_2(CSF_n)=\sum\limits_{1 \leq t,s \leq \chi(CSF_n)}^{t<s}ts\eta_{ts}= \frac{8(4n-7)}{3}+\frac{3(4n-10)}{3}+\frac{30(n-1)}{3}+84=\frac{74n+136}{3}.$$  

\noi \textit{Case-2.3:} Let $n \equiv -1 \Mod{3}$.Here we see that $\eta_{12}= \frac{4n-2}{3}$, $\eta_{23}=\eta_{13}=\frac{4n-5}{3}$, $\eta_{14}= \eta_{24}=\frac{n-2}{3}$, $\eta_{34}=\frac{n+1}{3}$, $\eta_{35}= 2$ and $\eta_{15}=\eta_{45}= 1$ . The definition of second chromatic Zagreb index gives the sum $$M^{\varphi^-}_2(CSF_n)=\sum\limits_{1 \leq t,s \leq\chi(CSF_n)}^{t<s}ts\eta_{ts}= \frac{2(4n-2)}{3}+\frac{9(4n-5)}{3}+\frac{12(n-2)}{3}+\frac{12(n+1)}{3}+55=\frac{68n+134}{3}.$$  

\noi \textit{Part (iii):} To find the minimum irregularity measurement, we consider the following cases, after applying the colouring mentioned above:

\textit{Case-3.1:} Let $n \equiv 0 \Mod{3}$.  Here, the result obviously follows from the following calculations: $$ M^{\varphi^-}_3(CSF_n)= \sum\limits_{i=1}^{n-1}\sum\limits_{j=2}^{n}|\zeta(v_i)-\zeta(v_j)|=\frac{16n}{3}+\frac{6n}{3}=\frac{22n}{3}.$$ 

\textit{Case-3.2:} Let $n \equiv 1 \Mod{3}$.  In this case, the result  follows from the following calculations: $$ M^{\varphi^-}_3(CSF_n)= \sum\limits_{i=1}^{n-1}\sum\limits_{j=2}^{n}|\zeta(v_i)-\zeta(v_j)|\\=\frac{2(4n-7)}{3}+\frac{2(4n-10)}{3}+\frac{9(n-1)}{3}+17=\frac{25n+8}{3}.$$

\textit{Case-3.3:} Let $n \equiv -1 \Mod{3}$.  Here the result  follows from the following calculations: $$ M^{\varphi^-}_3(CSF_n)=  \sum\limits_{i=1}^{n-1}\sum\limits_{j=2}^{n}|\zeta(v_i)-\zeta(v_j)|\\=\frac{(4n-2)}{3}+\frac{3(4n-5)}{3}+\frac{5(n-2)}{3}+\frac{n+1}{3}+9=\frac{22n+10}{3}.$$

\noi \textit{Part (iv):}  To calculate the total irregularity of $CSF_n$, all the possible vertex pairs from $CSF_n$ have to be considered and their possible colour distances are determined. The possibility of the vertex pairs which contribute to the colour distance can be classified according to the following three cases.
\textit{Case-1:} Let $n \equiv 0 \Mod{3}$.  The combinations possible are charted as $\{1,2\}, \{2,3\}, \{3,4\}$ contributing $1$, $\{1,3\}, \{2,4\}$ contributing $2$  and $\{1,4\}$ contributing $3$. Observe that $\theta(c_1)=\theta(c_2)=\theta(c_3)=\frac{2n}{3}$ and $\theta(c_4)= 1$. Thus, we have
\begin{eqnarray*}
M^{\varphi^-}_4(CSF_n)& = & \frac{1}{2}\sum\limits_{u,v \in V(CSF_n)}|\zeta(u)-\zeta(v)|\\ 
& = & \frac{16n^2+36n}{18}.
\end{eqnarray*}
\textit{Case-2:} Let $n \equiv 1 \Mod{3}$. Here the possible combinations which contributes to the colour distances are $\{1,2\}, \{2,3\}, \{3,4\},  \{4,5\}$ contributing $1$, $\{1,3\}, \{2,4\},  \{3,5\}$ contributing $2$,  $\{1,4\}, \{2,5\}$ contributing $3$and $\{1,5\}$ contributing $4$. We calculate the total irregularity as given below:

\begin{eqnarray*}
M^{\varphi^-}_4(CSF_n) & = & \frac{1}{2}\sum\limits_{u,v \in V(CSF_n)}|\zeta(u)-\zeta(v)|\\
& = & \frac{16n^2+94n-92}{18}
\end{eqnarray*} 
\textit{Case-3:} Let $n \equiv -1 \Mod{3}$. Here the possible combinations which contributes to the colour distances are $\{1,2\}, \{2,3\}, \{3,4\},  \{4,5\}$ contributing $1$, $\{1,3\}, \{2,4\},  \{3,5\}$ contributing $2$,  $\{1,4\}, \{2,5\}$ contributing $3$and $\{1,5\}$ contributing $4$. We calculate the total irregularity as given below:

\begin{eqnarray*}
M^{\varphi^-}_4(CSF_n) & = & \frac{1}{2}\sum\limits_{u,v \in V(CSF_n)}|\zeta(u)-\zeta(v)|\\
& = & \frac{16n^2+74n-32}{18}
\end{eqnarray*} 
\end{proof}

Using the minimum parameter colouring we can also work on $\varphi_+$ colouring  of closed sunflower graphs. Next theorem deals with this matter. 
\begin{theorem}\label{Thm-4.2}
For a closed sunflower graph $CSF_n$, we have
\begin{enumerate}
\item[(i)] $M_{1}^{\varphi^{-}}(CSF_n)= 
\begin{cases}
\frac{58n+3}{3}; &n \equiv 0 \Mod{3}\\
\frac{100n-73}{3}; &n \equiv 1 \Mod{3}\\
\frac{100n-35}{3}; &n \equiv -1 \Mod{3};
\end{cases}$
\item[(ii)] $M_{2}^{\varphi^{-}}(CSF_n)= 
\begin{cases}
\frac{113n}{3}; &n \equiv 0 \Mod{3}\\
\frac{200n-188}{3}; &n \equiv 1 \Mod{3}\\
\frac{212n-154}{3}; &n \equiv -1 \Mod{3};
\end{cases}$
\item[(iii)] $M_{3}^{\varphi^{-}}(CSF_n)= 
\begin{cases}
\frac{22n}{3}; &n \equiv 0 \Mod{3}\\
\frac{25n+8}{3}; &n \equiv 1 \Mod{3}\\
\frac{22n+10}{3}; &n \equiv -1 \Mod{3};
\end{cases}$
\item[(iv)] $M_{4}^{\varphi^{-}}(CSF_n)= 
\begin{cases}
\frac{16n^2+36n}{18}; &n \equiv 0 \Mod{3}\\
\frac{16n^2+94n-92}{18}; &n \equiv 1 \Mod{3}\\
\frac{16n^2+74n-32}{18}; &n \equiv -1 \Mod{3};
\end{cases}$
\end{enumerate}
\end{theorem}


\begin{proof}~
A closed sunflower graph $CSF_n$ has chromatic number $4$ when $n \equiv 0 \Mod{3}$ and has chromatic number $5$ otherwise. Let $v_1, v_2, \ldots v_n$ be the vertices of the inner wheel and $u_1, u_2, \ldots u_n$ be the vertices on the rim of the outer wheel and $v$ be the central vertex. In order to calculate the minimum values of chromatic Zagreb indices we apply the $\varphi^{-}$ colouring pattern to $CSF_n$ as described  below.\\ 
Let $n \equiv 0 \Mod{3}$ be assumed. When $n \equiv 0 \Mod{3}$, obeying the rules of maximum colouring, we can find three colour classes with same cardinality $\frac{2n}{3}$ and we colour them with maximal colours $c_4, c_3, c_2$. The central vertex $v$ is coloured with $c_1$.So we have $\theta(c_4)=\theta(c_3)=\theta(c_2)=\frac{2n}{3}$ and $\theta(c_1)= 1$. Also, $\eta_{24}= \eta_{23}=\eta_{34}=\frac{4n}{3}$ and $\eta_{12}= \eta_{13}=\eta_{14}=\frac{n}{3}$.\\
Let $n \equiv 1 \Mod{3}$ be assumed. Here we form three colour classes with the maximum independent sets having the same cardinality $\frac{2(n-1)}{3}$ and we colour them with maximal colours $c_5, c_4, c_3$. Also the vertices $u_1$ and $v_n$ are coloured with  $c_2$ and the central vertex $v$ is coloured with $c_1$. So here we have the values, $\theta(c_5)=\theta(c_4)=\theta(c_3)=\frac{2(n-1)}{3}$,  $\theta(c_2)= 2$, $\theta(c_1)= 1$ and $\eta_{45}= \eta_{34}=\frac{4n-7}{3}$, $\eta_{35}=\frac{4n-10}{3}$, $\eta_{13}= \eta_{14}=\eta_{15}=\frac{n-1}{3}$, $\eta_{23}= \eta_{25}=3$, $ \eta_{24}=2$  and $\eta_{12}=1$.\\
Let $n \equiv -1 \Mod{3}$ be assumed. Here we form three colour classes with the maximum independent sets having the same cardinality $\frac{2n-1}{3}$ and we colour them with maximal colours $c_5, c_4, c_3$. Also the balance two vertices $v$ and $v_n$ are coloured with  $c_2$ and $c_1$ respectively. Thus the following are the values: $\theta(c_5)=\theta(c_4)=\theta(c_3)=\frac{2n-1}{3}$,  $\theta(c_1)= \theta(c_2)= 1$, $\eta_{45}= \frac{4n-2}{3}$, $\eta_{34}=\eta_{35}=\frac{4n-5}{3}$, $\eta_{25}= \eta_{24}=\frac{n-2}{3}$, $\eta_{23}=\frac{n+1}{3}$, $\eta_{13}= 2$ and $\eta_{12}=\eta_{14}=\eta_{15}= 1$. \\
The balance proof follows exactly like that of just previous theorem.

\end{proof}
\section{Chromatic Topological Indices of Blossom Graphs}
A \textit{blossom graph} $Bl_n$ is the graph obtained by joining all vertices of the outer cycle of a closed sunflower graph $CSF_n$ to its central vertex. The chromatic Zagreb indices of blossom graphs are determined in the following results.
\begin{theorem}\label{Thm-5.1}
For the blossom graph $Bl_n= $, we have 
\begin{enumerate}
\item[(i)]~ $M_{1}^{\varphi^{-}}(Bl_n)=
\begin{cases}
\frac{30n+25}{2}; & \text{if $n$ is even}\\
\frac{15n+40}{2}; & \text{if $n$ is odd};
\end{cases}$
\item[(ii)]~ $M_{2}^{\varphi^{-}}(Bl_n)= 
\begin{cases}
\frac{95n}{2}; & \text{if $n$ is even}\\
\frac{99n-89}{2}; & \text{if $n$ is odd};
\end{cases}$
\item [(iii)]~ $M_{3}^{\varphi^{-}}(Bl_n)= 
\begin{cases}
\frac{22n}{2}; & \text{if $n$ is even}\\
\frac{25n-25}{2}; & \text{if $n$ is odd};
\end{cases}$
\item[(iv)]~ $M_{4}^{\varphi^{-}}(Bl_n)=
\begin{cases}
\frac{5n^2+10n}{8}; & \text{if $n$ is even}\\
\frac{11n^2+40n+29}{8}; & \text{if $n$ is odd}.
\end{cases}$
\end{enumerate}
\end{theorem}
\begin{proof}~
It is so clear that the blossom graph $Bl_n$ has chromatic number $5$. In $Bl_n$ let's put $v_1, v_2, \ldots v_n$ be the vertices on the outer cycle, $u_1, u_2, \ldots u_n$ be the vertices on the inner cycle and $u$ be the central vertex. Now we apply the $ \varphi^{-}$ colouring pattern to $Bl_n$ as described  below.

When $n$ is even, the outer cycle can be coloured with $c_1$ and $c_2$ alternatively and the inner cycle can be coloured with $c_3$ and $c_4$ alternatively such that both colour classes have cardinality $\frac{n}{2}$ .  We colour the central vertex $u$ with colour $c_5$. 

Now, let $n$ be odd. Here we colour the vertices $\{v_1, v_3, \ldots v_{n-2},u_{n-1} \}$ with colour  $c_1$ and $\{v_2, v_4, \ldots v_{n-1}, u_{n}\}$ with colour  $c_2$. The vertices $\{u_1, u_3, \ldots u_{n-2}, v_n \}$ with colour  $c_3$ and $\{u_2, u_4, \ldots u_{n-3}\}$ with colour  $c_4$. The central vertex is coloured with colour $c_5$. Now we proceed to the four parts of the theorem.

%
\noi \textit{Part (i):}  In order to find $M_{1}^{\varphi^{-}}$ of $Bl_n$, we first colour the vertices as mentioned above and then proceed to consider the following cases.

\textit{Case-1:} Let $n$ be even, then we have $\theta(c_1)=\theta(c_2)=\theta(c_3)=\theta(c_4)=\frac{n}{2}$ and $\theta(c_5)= 1$. Therefore, the corresponding chromatic topological index  is given by
$$M^{\varphi^-}_1(Bl_n)=\sum\limits_{i=1}^{5}(\zeta(v_i))^2=30n+25.$$

\noi \textit{Case-2:} Let $n$ be odd. Then, we have $\theta(c_1)=\theta(c_2)=\theta(c_3)=\frac{n+1}{2}$, $\theta(c_4)=\frac{n-3}{2}$ and $\theta(c_5)=1.$ Now, by the definition of first chromatic Zagreb index, we have $$M^{\varphi^-}_1(Bl_n)=\sum\limits_{i=1}^{5}(\zeta(v_i))^2=15n+40.$$
%
\noi \textit{Part (ii):} We colour the vertices as per the instructions in the introductory part for odd cases of $n$. For even case, the colouring is mentioned in Case 1. Now consider the following cases:

\noi \textit{Case- 1:} Let $n$ be even. In order to get the minimum value of all second Zagreb indices, we follow like this. The outer cycle can be coloured with $c_1$ and $c_4$ alternatively and the inner cycle can be coloured with $c_2$ and $c_3$ alternatively such that both colour classes have cardinality $\frac{n}{2}$ .  We colour the central vertex $u$ with colour $c_5$. Here we see that $\eta_{14}= \eta_{23}=n$and  $\eta_{12}=\eta_{13}= \eta_{15}=\eta_{25}=\eta_{34}=\eta_{35}=\eta_{45}=\eta_{24}=\frac{n}{2}$. 
The definition of second chromatic Zagreb index gives the sum $$M^{\varphi^-}_2(Bl_n)=\sum\limits_{1 \leq t,s \leq \chi(Bl_n)}^{t<s}ts\eta_{ts}=6n+4n+\frac{75n}{2}=\frac{95n}{2}.$$  
\noi \textit{Case- 2:} Let $n$ be odd. Here we see that $\eta_{12}=n+1, \eta_{13}=\eta_{23}=\frac{n+5}{2},\eta_{34}=n-3, \eta_{15}=\eta_{25}=\eta_{35}=\frac{n+1}{2}$and $ \eta_{14}=\eta_{24}=\eta_{45}=\frac{n-3}{2}$. Hence, we have the sum $$M^{\varphi^-}_2(Bl_n)=\sum\limits_{1 \leq t,s \leq \chi(Bl_n)}^{t<s}ts\eta_{ts}= \frac{99n-89}{2}.$$   
\noi \textit{Part (iii):} To find the minimum irregularity measurement, consider the following cases:

\noi \textit{Case- 1:} Let $n$ be even. Here we see that $\eta_{12}+ \eta_{23}+ \eta_{34}+ \eta_{45}$ edges contributes the distance $1$ to the total summation while $\eta_{13}+ \eta_{24}+ \eta_{35}$ contributes the distance $2$, $\eta_{14}+ \eta_{25}$ contributes the distance $3$ and $\eta_{15}$ contributes the distance $4$. The result follows from the following calculations:
$$ M^{\varphi^-}_3(Bl_n)= \sum\limits_{i=1}^{n-1}\sum\limits_{j=2}^{n}|\zeta(v_i)-\zeta(v_j)|=2n+9n=11n.$$ 
\noi \textit{Case- 2:} 
Let $n$ be odd. Here also we have that $\eta_{12}+ \eta_{23}+ \eta_{34}+ \eta_{45}$ edges contributes the distance $1$ to the total summation while $\eta_{13}+ \eta_{24}+ \eta_{35}$ contributes the distance $2$, $\eta_{14}+ \eta_{25}$ contributes the distance $3$ and $\eta_{15}$ contributes the distance $4$. Then the result follows from the following calculations: $$M^{\varphi^-}_3(Bl_n)= \sum\limits_{i=1}^{n-1}\sum\limits_{j=2}^{n}|\zeta(v_i)-\zeta(v_j)|=(n+1)+3(n-2)+4(n-3)+\frac{9(n+1)}{2}= \frac{25(n-1)}{2}.$$ 
\noi \textit{Part (iv):}  To calculate the total irregularity of $Bl_n$, all the possible vertex pairs from $Bl_n$ have to be considered and their possible colour distances are determined. The possibility of the vertex pairs which contribute to the colour distance can be classified according to the following two cases.

\noi \textit{Case- 1:} Let $n$ be even. The combinations possible are charted as $\{1,2\}, \{2,3\}, \{3,4\},  \{4,5\}$ contributing $1$, $\{1,3\}, \{2,4\},  \{3,5\}$ contributing $2$,  $\{1,4\}, \{2,5\}$ contributing $3$and $\{1,5\}$ contributing $4$. Observe that $\theta(c_1)=\theta(c_2)=\theta(c_3)=\theta(c_4)=\frac{n}{2}$ and $\theta(c_5)= 1$. Thus, we have
\begin{eqnarray*}
M^{\varphi^-}_4(Bl_n)& = & \frac{1}{2}\sum\limits_{u,v \in V(Bl_n)}|\zeta(u)-\zeta(v)|\\ 
& = & \frac{5n^2+10n}{8}.
\end{eqnarray*}

\noi \textit{Case- 2:} Let $n$ be odd. Here the possible combinations which contributes to the colour distances are $\{1,2\}, \{2,3\}, \{3,4\},  \{4,5\}$ contributing $1$, $\{1,3\}, \{2,4\},  \{3,5\}$ contributing $2$,  $\{1,4\}, \{2,5\}$ contributing $3$and $\{1,5\}$ contributing $4$. We calculate the total irregularity as given below:

\begin{eqnarray*}
M^{\varphi^-}_4(Bl_n) & = & \frac{1}{2}\sum\limits_{u,v \in V(Bl_n)}|\zeta(u)-\zeta(v)|\\
& = & \frac{11n^2+40n+29}{8}
\end{eqnarray*}

\end{proof}
Using the minimum parameter colouring we can also work on $\varphi_+$ colouring  of blossom graphs. Next theorem deals with this matter.

\begin{theorem}\label{Thm-5.2}
For a blossom graph $Bl_n $, we have 
\begin{enumerate}
\item[(i)]~ $M_{1}^{\varphi^{+}}(Bl_n)=
\begin{cases}
27n+1; & \text{if $n$ is even}\\
27n+20; & \text{if $n$ is odd};
\end{cases}$
\item[(ii)]~ $M_{2}^{\varphi^{+}}(Bl_n)= 
\begin{cases}
63n; & \text{if $n$ is even}\\
\frac{111n+91}{2}; & \text{if $n$ is odd};
\end{cases}$
\item [(iii)]~ $M_{3}^{\varphi^{+}}(Bl_n)= 
\begin{cases}
\frac{19n}{2}; & \text{if $n$ is even}\\
11n+1; & \text{if $n$ is odd};
\end{cases}$
\item[(iv)]~ $M_{4}^{\varphi^{+}}(Bl_n)=
\begin{cases}
\frac{11n^2+12n}{8}; & \text{if $n$ is even}\\
\frac{10n^2+16n-5}{8}; & \text{if $n$ is odd}.
\end{cases}$
\end{enumerate}
\end{theorem}

\begin{proof}~
The proof follows exactly as mentioned in the proof Theorem \ref{Thm-5.1}.
\end{proof}

\section{Conclusion}

The topic discussed in this paper do find a variety of applications in chemical graph theory and distribution theory. An outline of chromatic Zagreb indices and irregularity indices of some cycle related graphs are provided in this paper. The study seems to be promising for further studies as these indices can be computed for many graph classes and classes of derived graphs. The chromatic topological indices can be determined for graph operations, graph products and
graph powers. The study on the same with respect to different types of graph colourings also seem to be much promising. The concept can be extended to edge colourings and map colourings also. In Chemistry, some interesting studies using the above-mentioned concepts are possible if $c(v_i)$ (or $\zeta(v_i)$) assumes the values such as energy, valency, bond strength etc. Similar studies are possible in various other fields. All these facts highlight the wide scope for further research in this area. Even the chromatic version of other topological indices gives new areas of research with massive applications.

\section*{Acknowledgement}

The first author would like to acknowledge the academic helps rendered by Centre for Studies in Discrete Mathematics, Vidya Academy of Science and Technology, Thrissur, Kerala, India.

\end{document}